\documentclass[12pt]{article}
\usepackage{mathrsfs}
\usepackage{bbm}

\usepackage{amssymb,a4}
\usepackage{amsmath,amsfonts,amssymb,amsthm}

\newcommand{\1}{\mathbbm {1}}
\newcommand{\Z}{{\mathbb Z}}
\newcommand{\C}{{\mathbb C}}
\newcommand{\h}{{\mathfrak h}}
\newcommand{\wh}{{\widehat{\mathfrak h}}}
\newcommand{\I}{{\mathcal I}}
\newcommand{\tI}{\widetilde{\I}}
\newcommand{\J}{{\mathcal J}}

\def\<{\langle}
\def\>{\rangle}
\def\a{\alpha}
\def\b{\beta}
\def\g{\mathfrak g}
\def\h{\mathfrak h}

\newcommand{\mraff}{\mathrm{aff}}

\newcommand{\la}{\langle}
\newcommand{\ra}{\rangle}

\newtheorem{thm}{Theorem}[section]
\newtheorem{prop}[thm]{Proposition}
\newtheorem{lem}[thm]{Lemma}

\newtheorem{rmk}[thm]{Remark}

\begin{document}

\begin{center}
{\Large {\bf  The structure of parafermion vertex operator algebras: general case}} \\
\vspace{0.5cm}

%\title
%{The Parafermion vertex operator algebras of $ADE$ types}

Chongying Dong\footnote{Supported by NSF grants,
and  a Faculty research grant from  the
University of California at Santa Cruz.}
\\
Department of Mathematics, University of
California, Santa Cruz, CA 95064\\
\vspace{.1cm}
\& School of Mathematics,  Sichuan University, Chengdu, 610065 China\\
\vspace{.2 cm}
 Qing Wang\footnote{Supported by Natural Science Foundation of
Fujian Province, China (No.2009J05012.)}\\
School of Mathematical Sciences, Xiamen University,  Xiamen,
361005 China
\end{center}

%\author[C. Dong]{Chongying Dong}
%\address{Department of Mathematics, University of California, Santa Cruz,
%CA 95064 \& School of Mathematics, Sichuan University, China} \email{dong@math.%ucsc.edu}

%\author[Q. Wang]{Qing Wang}
%\address{School of Mathematics, Xiamen University, Xiamen, China}
%\email{qingwang@xmu.edu.cn}

%\subjclass[2000]{17B69, 17B65}

\begin{abstract}
The structure of the parafermion vertex operator algebra associated to an
integrable highest weight module for any affine Kac-Moody algebra
is studied. In particular, a set of generators for this algebra
has been determined.
\end{abstract}

%\maketitle

\section{Introduction}
\def\theequation{1.\arabic{equation}}
\setcounter{equation}{0}

This paper is a continuation of our study of the parafermion
vertex operator algebra associated to an integrable highest weight
module for an arbitrary  affine Kac-Moody algebra. We determine a
set of generators for these algebras. If the affine
Kac-Moody algebra is $A_1^{(1)}$, this result was obtained
previously in \cite{DLY3}.

The parafermion algebra was first studied in \cite{ZF} in the context
of conformal field theory.
 It was clarified in \cite{DL} that the
parafermion algebras are essentially the $Z$-algebras introduced
and studied earlier in \cite{LP,LW1,LW2} in the process of
studying the representation theory for the affine Kac-Moody Lie
algebras. As it proved in \cite{DL}, the parafermion algebras
generate certain generalized vertex operator algebras. The
partition functions for the parafermion conformal field theory
have been given in \cite{GQ} and \cite{G} in connection with the
partition functions associated to the integrable representations
for the affine Kac-Moody Lie algebras. We refer the reader to
\cite{ZF}, \cite{GQ}, \cite{G}, \cite{BEHHH}, \cite{H}, \cite{WD},
\cite{We} for various aspects of parafermion conformal field
theory.

The parafermion vertex operator algebras which have roots in the
parafermion conformal field theory are realized as commutants of the
Heisenberg vertex operator subalgebras in the vertex operator
algebras associated to the integrable highest weight modules for the
affine Kac-Moody Lie algebra \cite{GKO}, \cite{FZ}, \cite{LL}. More
precisely, let $L(k,0)$ be the level $k$ integrable highest weight
module for  affine Kac-Moody algebra $\widehat{\g}$ associated to a
finite dimensional simple Lie algebra $\g.$ Then $L(k,0)$ has a
vertex operator subalgebra $M_{\wh}(k,0)$ generated by the Cartan
subalgebra $\h$  of $\g.$  The commutant $K(\g,k)$ of $M_{\wh}(k,0)$
in $L(k,0)$ is called the parafermion vertex operator algebra.
Although the parafermion field theory has been studied for more than
two decades, the mathematical investigation of the parafermion
vertex operator algebras have been limited by a lack of
understanding of the structural theory of these algebras. The  goals
of this paper and \cite{DLY1}-\cite{DLY3} are to alleviate this
situation.

More importantly, it is widely believed that $K(\g,k)$ should give
a new class of rational, $C_2$-cofinite vertex operator algebras
although this can only been proved in the case $\g=sl_2$ and
$k\leq 6$ \cite{DLY2}.  Most well known vertex operator algebras
such as lattice vertex operator algebras \cite{B}, \cite{FLM},
\cite{D}, the affine vertex operator algebras \cite{FZ},
\cite{DL}, \cite{MP} and the Virasoro vertex operator algebras
\cite{FZ}, \cite{W} can be understood well using the underline
lattices or Lie algebras. Unfortunately, the structures of
parafermion vertex operator algebras with weight one subspaces
being zero are much more complicated. It seems that a
determination of a set of generators is the first step in
understanding parafermion vertex operator algebras and their
representation theory.

It is well known that $L(k,0)$ is the irreducible quotient of the
generalized Verma module $V(k,0)$ (see Section 2). So the structure
of $L(k,0)$ can be determined by studying the maximal submodule of
$V(k,0)$ or the maximal ideal of $V(k,0)$ which is also a vertex
operator algebra. The same idea can also been applied to the study
of parafermion vertex operator algebras. In fact, the Heisenberg
vertex operator algebra $M_{\wh}(k,0)$ is also a subalgebra of
$V(k,0)$ and the parafermion vertex operator algebra $K(\g,k)$ is
the simple quotient of the commutant $N(\g,k)$ of $M_{\wh}(k,0)$ in
$V(k,0).$ The main part of this paper is to determine a set of
generators for $V(k,0)(0)$ which is the weight zero subspace of
$V(k,0)$ under the action of Cartan subalgebra $\h.$ The generators
for $N(\g,k)$ and $K(\g,k)$ will be found easily then.  Also, the
maximal ideal of $N(\g,k)$ is generated by one vector. This result
is similar to that for the maximal ideal of $V(k,0).$

It is worthy to point out that the structure theory for the
parafermion vertex operator algebra is similar to the structure
theory for the finite dimensional Lie algebras or Kac-Moody Lie
algebras. The building block of the Kac-Moody Lie algebras is the
3-dimensional simple Lie algebra $sl_2$ associated to any real root.
The generator results for the parafermion vertex operator algebras
given in this paper and \cite{DLY3} show that the parafermion vertex
operator algebras associated to the affine Lie algebra $A_1^{(1)}$
are  also the building block of general parafermion vertex operator
algebras. We hope this fact will be important in the future study of
the representation theory for parafermion vertex operator algebra.
So a completely understanding of representation theory of $K(\g,k)$
in the case $\g=sl_2$ becomes necessary.

The paper is organized as follows. In Section 2, we give the
construction of vertex operator algebra $V(k,0)$ associated to the
affine Kac-Moody algebra $\widehat{\g}$ from \cite{FZ}. $V(k,0)$
has a vertex operator subalgebra $V(k,0)(0)$ which is the
$\h$-invariants of $V(k,0).$ A foundational and difficult result
in this section is a set of of generators for $V(k,0)(0).$ In
Section 3, we give a set of generators for the vertex operator
algebra $N(\g,k)$ which is the commutant of the Heisenberg vertex
operator algebra $M_{\wh}(k,0)$ in $V(k,0).$ We also discuss the
vertex operator subalgebra $\widehat{ P}_{\alpha}$ generated by
$\omega_{\alpha}, W_\a^3$ associated to any positive root $\alpha$
and prove that $\widehat{P}_{\alpha}$ is isomorphic to
$N(sl_2,k_\alpha)$, where
$k_\alpha=\frac{\<\theta,\theta\>}{\<\a,\a\>}k $ and $\theta$ is
the highest root. In fact, $N(\g,k)$ is generated by
$\widehat{P}_{\alpha}$ for positive roots $\alpha.$ In Section 4,
we give a set of generators for the parafermion vertex operator
algebra $K(\g,k)$ which is the simple quotient of $N(\g,k).$ We
prove that the maximal ideal of $N(\g,k)$ is generated by the
vector $x_{-\theta}(0)^{k+1}x_{\theta}(-1)^{k+1}\1.$ We also show
that the image $P_{\a}$  of $\widehat{P}_{\a}$ in $K(\g,k)$ is
isomorphic to $K(sl_2,k_{\alpha})$ and $K(\g,k)$ is generated by
$P_{\a}$ for positive roots $\alpha.$ That is,
$K(sl_2,k_{\alpha})$ are the building blocks of $K(\g,k).$

We expect the reader to be familiar with the elementary theory of vertex
operator algebras as found, for example, in \cite{FLM} and \cite{LL}
.
\section{Vertex operator algebras $V(k,0)$ and $V(k,0)(0)$}
\label{Sect:V(k,0)}
\def\theequation{2.\arabic{equation}}
\setcounter{equation}{0}

Let $\g$ be a finite dimensional simple Lie algebra with a Cartan
subalgebra $\h.$ Let $\Delta$ be the corresponding root system and
$Q$ the root lattice. Let  $\la ,\ra$ be an invariant symmetric
nondegenerate bilinear form on $\g$ such that $\<\a,\a\>=2$ if
$\alpha$ is a long root, where we have identified $\h$ with $\h^*$
via $\<,\>.$ As in \cite{H}, we denote the image of $\alpha\in
\h^*$ in $\h$ by $t_\alpha.$ That is, $\alpha(h)=\<t_\alpha,h\>$
for any $h\in\h.$ Fix simple roots $\{\alpha_1,\cdots,\alpha_l\}$
and denote the highest root by $\theta.$

Let $\g_{\alpha}$ denote the root space associated to the root
$\a\in \Delta.$ For $\alpha\in \Delta_+$, we fix $x_{\pm
\alpha}\in \g_{\pm \alpha}$ and
$h_{\alpha}=\frac{2}{\<\a,\a\>}t_\alpha\in \h$ such that
 $[x_\a,x_{-\a}]=h_{\a}, [h_\a,x_{\pm \a}]=\pm 2x_{\pm\a}.$ That
is, $\g^{\a}=\C x_{\a}+\C h_{\alpha}+\C x_{-\alpha}$ is isomorphic
to $sl_2$ by sending $x_\a$ to $\left(\begin{array}{ll} 0 & 1\\ 0 &
0\end{array}\right),$ $x_{-\a}$ to $\left(\begin{array}{ll} 0 & 0\\
1 & 0\end{array}\right)$ and $h_\a$ to $\left(\begin{array}{ll} 1 &
0\\ 0 & -1\end{array}\right).$ Then
$\<h_\a,h_\a\>=2\frac{\<\theta,\theta\>}{\<\alpha,\alpha\>}$ and
$\<x_{\a},x_{-\a}\>=\frac{\<\theta,\theta\>}{\<\alpha,\alpha\>}$
for all
$\alpha\in \Delta.$

%Let $E_{ij}(i,j=1,\ldots,l+1)$ denote the standard basis of the space of all
%$(l+1)\times(l+1)$ matrices. Let $\mathfrak h$ be the space of all
%traceless diagonal matrices. Then
%$$\alpha_i^{\vee}=E_{ii}-E_{i+1,i+1} (i=1,\ldots,l)$$ form a
%basis of $\mathfrak h$. Define $\epsilon_{i}\in {\mathfrak
%h}^{\ast}(i=1,\ldots,l+1)$ by
%$$\epsilon_{i}(\mbox{diag}(a_{1},\ldots,a_{l+1}))=a_{i}.$$
%Then $$\alpha_{i}=\epsilon_{i}-\epsilon_{i+1}(i=1,\ldots,l)$$ form
%a basis of ${\mathfrak h}^{\ast}$.
%$\Delta=\{\epsilon_{i}-\epsilon_{j}|\; 1\leq i\neq j\leq l+1\}$ is
%the set of all roots, $\Delta_{+}=\{\epsilon_{i}-\epsilon_{j}|\;
%1\leq i< j\leq l+1\}$ is the set of positive roots. The root space
%decomposition with respect to $\mathfrak h$ is
%$$sl_{l+1}=\mathfrak h\bigoplus\bigoplus_{\alpha\in\Delta}\C x_{\alpha}$$
%where $x_{\e_i-\e_j}=E_{ij}$ for root $\e_i-e_j.$
%We define a nondegenerate symmetric invariant bilinear form on
% $sl_{l+1}$ as follows: $\la A,B\ra=tr AB$ for $A,B\in  sl_{l+1}.$
%Then $\la \alpha_{i}^{\vee},\alpha_{i}^{\vee} \ra
%= 2$, $\la x_{\alpha}, x_{-\alpha} \ra=1\ (\alpha\in\Delta).$
%We identify ${\mathfrak h}^{\ast}$ with $\mathfrak h$
%using the form $\la,\ra.$ With this identification,
%$\alpha_{i}$ corresponds to $\alpha^{\vee}_{i}.$
Let $\widehat{\mathfrak g}= \g \otimes
\C[t,t^{-1}] \oplus \C K$ is the corresponding affine Lie algebra.
Let $k \ge 1$ be an integer and
\begin{equation*}
V(k,0) = V_{\widehat{\g}}(k,0) = Ind_{\g \otimes
\C[t]\oplus \C K}^{\widehat{\g}}\C
\end{equation*}
the induced $\widehat{\g}$-module such that ${\g} \otimes \C[t]$ acts as $0$ and $K$ acts as $k$ on $\1=1$.

We denote by $a(n)$ the operator on $V(k,0)$ corresponding to the action of
$a \otimes t^n$. Then
%\begin{equation}\label{eq:affine-commutation}
$$[a(m), b(n)] = [a,b](m+n) + m \la a,b \ra \delta_{m+n,0}k$$
%\end{equation}
for $a, b \in \g$ and $m,n\in \Z$.

Let $a(z) = \sum_{n \in \Z} a(n)z^{-n-1}$. Then $V(k,0)$ is a
vertex operator algebra generated by $a(-1)\1$ for $a\in \g$ such
that $Y(a(-1)\1,z) = a(z)$ with the vacuum vector $\1$ and the
Virasoro vector
\begin{align*}
\omega_{\mraff} &= \frac{1}{2(k+h^{\vee})} \Big(
\sum_{i=1}^{l}h_i(-1)h_i(-1)\1 +\sum_{ \alpha\in\Delta}
\frac{\<\a,\a\>}{\<\theta,\theta\>}x_{\alpha}(-1)x_{-\alpha}(-1)\1
\Big)
\end{align*}
of central charge $\frac{k\dim \g}{k+h^{\vee}}$ \cite{FZ} (also see
\cite[Section 6.2]{LL}), where $h^{\vee}$ is the dual Coxeter number
of $\g$ and $\{h_i|i=1,\cdots,l\}$ is an orthonormal basis of
$\mathfrak h.$

For $\lambda \in {\mathfrak h}^*$, set
\begin{equation*}
V(k,0)(\lambda)=\{v\in V(k,0)|h(0)v=\lambda(h) v, \forall\;
h\in\mathfrak h\}.
\end{equation*} Then we have
\begin{equation}\label{eq:V-dec}
V(k,0)=\oplus_{\lambda\in Q}V(k,0)(\lambda).
\end{equation}

Since $[h(0), Y(u,z)]=Y(h(0)u,z)$ for $h\in \mathfrak h$ and  $u
\in V(k,0)$, from the definition of affine vertex operator
algebra, we see that $V(k,0)(0)$ is a vertex operator subalgebra
of $V(k,0)$ with the same Virasoro vector $\omega_{\mraff}$ and
each $V(k,0)(\lambda)$ is a module for $V(k,0)(0)$.

Our first theorem is on a set of generators for $V(k,0)(0).$

\begin{thm}\label{generator1} The vertex operator algebra
$V(k,0)(0)$ is generated by vectors $\alpha_{i}(-1)\1$ and $x_{-\alpha}(-2)x_{\alpha}(-1)\1$ for
$1\leq i \leq l, \alpha\in\Delta_{+}$.
\end{thm}

\begin{proof} First note that $V(k,0)(0)$ is spanned
by the vectors
%\begin{equation}\label{eq:span}
$$a_1(-m_1)\cdots a_s(-m_s)x_{\beta_{1}}(-n_{1})x_{\beta_{2}}(-n_{2})\cdots
x_{\beta_{t}}(-n_{t})\1$$
%\end{equation}
where $a_i\in \mathfrak h, \beta_j\in\Delta, m_i>0, n_j>0$
and $\beta_{1}+\beta_{2}+\cdots+\beta_{t}=0.$
Let $U$ be the vertex operator subalgebra generated by
$\alpha_{i}(-1)\1$ and $x_{-\alpha}(-2)x_{\alpha}(-1)\1$
for $1\leq i \leq l, \alpha\in\Delta_{+}$. Clearly,
$\alpha_{i}(-1)\1$ and $x_{-\alpha}(-2)x_{\alpha}(-1)\1\in V(k,0)(0)$ for
$1\leq i \leq l, \alpha\in\Delta_{+}$. It suffices to prove that
$V(k,0)(0)\subset U.$

Since $(h(-1)\1)_n = h(n)$ for $h\in \mathfrak h $, we see that
$h(n)U \subset U$ for $h\in \mathfrak h$ and $n\in\Z.$ So we only
need to prove $u=x_{\beta_{1}}(-n_{1})x_{\beta_{2}}(-n_{2})\cdots
x_{\beta_{t}}(-n_{t})\1\in U$ with
$\beta_{1}+\beta_{2}+\cdots+\beta_{t}=0$.  We will prove by
induction on $t$.

Clearly, $t\geq 2.$ If $t=2,$ it follows from Theorem 2.1 in
[DLWY] that $x_{-\alpha}(-m)x_{\alpha}(-n)\1 \in U$ for $m,n>0$.
From now on, we assume that $t>2$ and that
$$x_{\beta_{1}}(-n_{1})x_{\beta_{2}}(-n_{2})\cdots
x_{\beta_{\nu}}(-n_{\nu})\1\in U$$ with
$\beta_{1}+\beta_{2}+\cdots+\beta_{\nu}=0$ for $2\leq \nu \leq
t-1$ and $n_i>0.$ We have to show that
$$x_{\beta_{1}}(-n_{1})x_{\beta_{2}}(-n_{2})\cdots
x_{\beta_{t}}(-n_{t})\1\in U$$ with
$\beta_{1}+\beta_{2}+\cdots+\beta_{t}=0$. We divide the proof into
two cases.

{\bf Case 1.} There exist $1\leq i,j\leq t$ such that
$\beta_{i}+\beta_{j}\in \Delta.$ We first prove the result for $n_2=1.$ Note
that if
 $$x_{\beta_{1}}(-n_{1})x_{\beta_{2}}(-n_{2})\cdots
x_{\beta_{t}}(-n_{t})\1\in U$$
 then
$$x_{\beta_{i_1}}(-n_{i_1})x_{\beta_{i_2}}(-n_{i_2})\cdots
x_{\beta_{i_t}}(-n_{i_t})\1\in U$$ by the induction assumption,
where $(i_1,...,i_t)$ is any permutation of $(1,...,t).$ Without
loss of generality, we may assume that $\beta_{1}+\beta_{2}\in
\Delta$.

We need the
following identity
%\begin{equation}\label{eq:gene}
$$(u_nv)_m=\sum_{j\geq 0}(-1)^j \binom{n}{j} u_{n-j}v_{m+j}
-\sum_{j\geq 0}(-1)^{n+j} \binom{n}{j} v_{m+n-j}u_{j}$$
%\end{equation}
for $u,v\in V(k,0)$ and $m,n\in \Z$, which is a consequence of the
Jacobi identity of vertex operator algebra. Now, take
$u=x_{-\alpha}(-1)\1, \; v=x_{\alpha}(-n)\1$ to obtain
\begin{equation}\label{eq:fel}
\begin{split}
& \qquad \quad (x_{-\alpha}(-m)x_{\alpha}(-n)\1)_{s}\\
& \qquad = \sum_{ j \geq 0}(-1)^{j}\binom{-m}{j}x_{-\alpha}(-m-j)(x_{\alpha}(-n)\1)(s+j)\\
& \qquad \quad  -\sum_{j\geq
0}(-1)^{-m+j}\binom{-m}{j}(x_{\alpha}(-n)\1)(s-m-j)x_{-\alpha}(j).
\end{split}
\end{equation}
Since $Y(u_{-n-1}\1,z)=\frac{1}{n!}\frac{d^{n}}{dz^{n}}Y(u,z)$, we
have
\begin{equation}\label{eq:fe11}
(x_{\alpha}(-n)\1)(s+j)=(-1)^{n-1}{s+j\choose n-1}x_{\alpha}(s+j+1-n)
\end{equation}
and
\begin{equation}\label{eq:fe12}
(x_{\alpha}(-n)\1)(s-m-j)=(-1)^{n-1}{s-m-j\choose n-1}x_{\alpha}(s-m-j+1-n).
\end{equation}

Let $w=x_{\beta_{1}+\beta_{2}}(-n_{2})x_{\beta_{3}}(-n_{3})\cdots
x_{\beta_{t}}(-n_{t})\1$ with
$\beta_{1}+\beta_{2}+\cdots+\beta_{t}=0.$ By
(\ref{eq:fel})-(\ref{eq:fe12}) we have
\begin{equation}\label{eq:fe13}
\begin{split}
& \qquad \quad (x_{\beta_{1}}(-m)x_{-\beta_{1}}(-n)\1)_{s}w\\
& \qquad = \sum_{ j \geq
0}(-1)^{j+n-1}\binom{-m}{j}\binom{s+j}{n-1}x_{\beta_{1}}(-m-j)x_{-\beta_{1}}(s+j+1-n)w\\
& \qquad \quad  -\sum_{j\geq
0}(-1)^{-m+j+n-1}\binom{-m}{j}\binom{s-m-j}{n-1}x_{-\beta_{1}}(s-m-j+1-n)x_{\beta_{1}}(j)w.
\end{split}
\end{equation}

Without loss, we assume $\beta_{i}-\beta_{1}\in \Delta$ for $3\leq
i\leq p\leq t$, $\beta_{i}-\beta_{1}\notin \Delta$ for $p+1\leq
i\leq q\leq t.$
%, and $\beta_{i_j}+\beta_{1}\in \Delta$ for $j\leq
%q,$ $\beta_{i_j}+\beta_{1}\notin \Delta$ for $q+1\leq j\leq t$ where
%$(j_3,...,j_t)$ is a permutation of $(3,...,t).$ Let
%$w'=x_{\beta_{1}+\beta_{2}}(-n_{2})x_{\beta_{i_3}}(-n_{i_3})\cdots
%x_{\beta_{i_t}}(-n_{i_t})\1.$
We now analyze the identity \eqref{eq:fe13} with
$-n_{2}=-1-n_{3}-\cdots-n_{p}$ and $s-n=n_{3}+\cdots+n_{p}-1.$
Consider the first summation of the right hand side of
\eqref{eq:fe13}. If $j=0$, then $s+j+1-n-1-n_{3}-\cdots-n_{p}=-1$
and $s+j+1-n-n_{i}=n_{3}+\cdots+n_{i-1}+n_{i+1}+\cdots+n_{p}\geq
0$ for $3\leq i\leq p.$ If $j>0,$ then
$s+j+1-n-1-n_{3}-\cdots-n_{p}=-1+j\geq0$, and
$s+j+1-n-n_{i}=n_{3}+\cdots+n_{i-1}+n_{i+1}+\cdots+n_{p}+j>0$ for
$3\leq i\leq p.$ Using the induction assumption, we see that
$$\sum_{ j >0}(-1)^{j+n-1}\binom{-m}{j}\binom{s+j}{n-1}x_{\beta_{1}}(-m-j)x_{-\beta_{1}}(s+j+1-n)w=u\in U.$$
Let $[x_{-\b_1},x_{\b_1+\b_2}]=\lambda x_{\b_2}$ for some nonzero
$\lambda.$ Then the identity \eqref{eq:fe13} becomes
\begin{equation}\label{eq:fe14}
\begin{split}
&  \qquad \quad (x_{\beta_{1}}(-m)x_{-\beta_{1}}(-n)\1)_{s}w\\
& \qquad
=\lambda(-1)^{n-1}\binom{s}{n-1}x_{\beta_{1}}(-m)x_{\beta_{2}}(-1)x_{\beta_{3}}(-n_{3})\cdots x_{\beta_{t}}(-n_{t})\1+u\\
 & \qquad  -\sum_{j\geq
0}(-1)^{-m+j+n-1}\binom{-m}{j}\binom{s-m-j}{n-1}(\lambda_2x_{-\beta_{1}}(s-m-j+1-n)\cdot\\
& \qquad \qquad x_{2\beta_1+\b_2}(-n_2+j)\prod_{i=3}^tx_{\beta_i}(-n_i)\1\\
&
\qquad+\lambda_{3}x_{-\beta_{1}}(s-m-j+1-n)x_{\beta_{1}+\beta_{2}}(-n_{2})x_{\beta_1+\b_3}(-n_3+j)\prod_{i=4}^tx_{\b_i}(-n_i)\1+\cdots\\
\\
& \qquad +
\!\lambda_{t}x_{-\beta_{1}}\!(\!s\!-\!m\!-\!j\!+\!1\!-\!n\!)x_{\beta_{1}\!+\!\beta_{2}}(\!-\!n_{2}\!)x_{\b_3}(-n_3)\cdots
x_{\beta_{t-1}}(-n_{t-1})x_{\b_1+\b_t}(-n_t+j)\1 )
\end{split}
\end{equation}
where constant $\lambda_i$ is determined by $[x_{\beta_1},
x_{\beta_i}]=\lambda_ix_{\beta_1+\beta_i}$ if $\b_1+\b_i\in
\Delta$, and $\lambda_i=0$ if $\b_1+\b_i\notin \Delta.$ The
$x_{\b}$ is understood to be zero if $\b$ is not a root.
 Take $N=\mbox{max}\{n_{2},\cdots,n_{t}\}.$
Then if $j\geq N,$ the corresponding term in the summation of
(\ref{eq:fe14}) lies in $U$ by the induction assumption. By
changing $u$ in (\ref{eq:fe14}), we can assume that the $j$ ranges
from $0$ to $N-1$ in (\ref{eq:fe14}).

Since $s>n-1$ we see that $\binom{s}{n-1}\ne 0.$ To prove
$$x_{\beta_{1}}(-m)x_{\beta_{2}}(-1)x_{\beta_{3}}(-n_{3})\cdots x_{\beta_{t}}(-n_{t})\1\in U,$$
it is enough to prove that each term in the summation
$\sum_{j=0}^{N-1}*$ of (\ref{eq:fe14}) lies in $U$ as the LHS of
(\ref{eq:fe14}) is always in $U.$ For this purpose, we return to
the identity \eqref{eq:fe13}. We assume
$-n_{2}=-1-n_{3}-\cdots-n_{p}$ as before. Let $r\geq 0$ be an
integer, and assume $s'-n=n_{3}+\cdots+n_{p}+r.$ Then for $j\geq
0$, we have $s'+j+1-n-1-n_{3}-\cdots-n_{p}=j+r\geq 0$, and
$s'+j+1-n-n_{i}=n_{3}+\cdots+n_{i-1}+n_{i+1}+\cdots+n_{p}+1+j+r>0$
for $3\leq i\leq p.$ After substituting $m+r+1$ for $m$ in
\eqref{eq:fe13} and $s'$ for $s,$ the identity \eqref{eq:fe13}
becomes
\begin{equation}\label{eq:fe15}
\begin{split}
& \qquad \quad (x_{\beta_{1}}(-m-r-1)x_{-\beta_{1}}(-n)\1)_{s'}w=u'\\
& \qquad
-\sum_{j=0}^{N-1}(\!-\!1\!)^{-\!m\!-\!r\!-\!1\!+\!j\!+\!n\!-\!1}\binom{-\!m\!-\!r\!-\!1}{j}\binom{s'\!-\!m\!-\!r\!-\!1\!-\!j}{n-1}(
\lambda_2x_{-\beta_{1}}(s'\!-\!m\!-\!r-\!j\!-\!n)\cdot\\
& \qquad \qquad x_{2\beta_1+\b_2}(-n_2+j)\prod_{i=3}^tx_{\beta_i}(-n_i)\1\\
&
\qquad+\lambda_{3}x_{-\beta_{1}}(s'\!-\!m\!-\!r-\!j\!-\!n)x_{\beta_{1}+\beta_{2}}(-n_{2})x_{\beta_1+\b_3}(-n_3+j)\prod_{i=4}^tx_{\b_i}(-n_i)\1+\cdots\\
\\
& \qquad +
\!\lambda_{t}x_{-\beta_{1}}\!(s'\!-\!m\!-\!r-\!j\!-\!n)x_{\beta_{1}\!+\!\beta_{2}}(\!-\!n_{2}\!)x_{\b_3}(-n_3)\cdots
x_{\beta_{t-1}}(-n_{t-1})x_{\b_1+\b_t}(-n_t+j)\1)
\end{split}
\end{equation}
for some $u^{'}\in U$ by the induction assumption.

Note that $s-m-j=s'-m-r-1-j.$ So \eqref{eq:fe15} implies to
\begin{equation}\label{eq:fe15'}
\begin{split}
& \qquad \quad (x_{\beta_{1}}(-m-r-1)x_{-\beta_{1}}(-n)\1)_{s'}w=u'\\
& \qquad -\sum_{j=0}^{N-1}(\!-\!1\!)^{-\!m\!-\!r\!-\!1\!+\!j\!+\!n\!-\!1}\binom{-\!m\!-\!r\!-\!1}{j}\binom{s-m-j}{n-1}(
\lambda_2x_{-\beta_{1}}(s-m-j+1-n)\cdot\\
& \qquad \qquad x_{2\beta_1+\b_2}(-n_2+j)\prod_{i=3}^tx_{\beta_i}(-n_i)\1\\
&
\qquad+\lambda_{3}x_{-\beta_{1}}(s-m-j+1-n)x_{\beta_{1}+\beta_{2}}(-n_{2})x_{\beta_1+\b_3}(-n_3+j)\prod_{i=4}^tx_{\b_i}(-n_i)\1+\cdots\\
\\
& \qquad +
\!\lambda_{t}x_{-\beta_{1}}\!(\!s\!-\!m\!-\!j\!+\!1\!-\!n\!)x_{\beta_{1}\!+\!\beta_{2}}(\!-\!n_{2}\!)x_{\b_3}(-n_3)\cdots
x_{\beta_{t-1}}(-n_{t-1})x_{\b_1+\b_t}(-n_t+j)\1 )
\end{split}
\end{equation}

We now take $r=0,1,2,\cdots,N-1$ in \eqref{eq:fe15'} and get a linear system with coefficient matrix
$$\left((-1)^{-m-r+j+n}\binom{-m-r-1}{j}\right)_{r,j=0,...,N-1}.$$
To prove each term in $\sum_{j=0}^{N-1}*$ lies in $U$, it is
enough to show that the coefficient matrix is nondegenerate. Note
that $m,n$ are fixed integers, it suffices to verify the matrix
$$\left((-1)^{r+j}\binom{-m-r-1}{j}\right)_{r,j=0,...,N-1}$$
 is nondegenerate. But this is clear. Using
 \eqref{eq:fe14}, we immediately see that
$$x_{\beta_{1}}(-m)x_{\beta_{2}}(-1)x_{\beta_{3}}(-n_{3})\cdots
x_{\beta_{k}}(-n_{k})\1\in U.$$

Next we prove that
$x_{\beta_{1}}(-m)x_{\beta_{2}}(-2)x_{\beta_{3}}(-n_{3})\cdots
x_{\beta_{k}}(-n_{k})\1\in U.$ The proof is similar. In this case
we let $-n_{2}=-2-n_{3}-\cdots-n_{p}$ and
$s-n=n_{3}+\cdots+n_{p}-1.$ Then for $j=0$, we have
$s+j+1-n-2-n_{3}-\cdots-n_{p}=-2$, and
$s+j+1-n-n_{i}=n_{3}+\cdots+n_{i-1}+n_{i+1}+\cdots+n_{p}\geq 0$
for $3\leq i\leq p.$ For $j=1,$ we have
$s+j+1-n-2-n_{3}-\cdots-n_{p}=-1$. For $j>1,$ we have
$s+j+1-n-2-n_{3}-\cdots-n_{p}=-2+j\geq0$, and for $j>0$,
$s+j+1-n-n_{i}=n_{3}+\cdots+n_{i-1}+n_{i+1}+\cdots+n_{p}+j>0$ for
$3\leq i\leq p.$ Note that
$x_{\beta_{1}}(-m)x_{\beta_{2}}(-1)x_{\beta_{3}}(-n_{3})\cdots
x_{\beta_{k}}(-n_{k})\1\in U.$ Similar to \eqref{eq:fe14}, we have
the following identity from \eqref{eq:fe13}:
\begin{equation}\label{eq:fe16}
\begin{split}
&  \qquad \quad (x_{\beta_{1}}(-m)x_{-\beta_{1}}(-n)\1)_{s}w\\
& \qquad
=\lambda(-1)^{n-1}\binom{s}{n-1}x_{\beta_{1}}(-m)x_{\beta_{2}}(-2)x_{\beta_{3}}(-n_{3})\cdots
x_{\beta_{t}}(-n_{t})\1+u\\
 & \qquad  -\sum_{j=0}^{N-1}(-1)^{-m+j+n-1}\binom{-m}{j}\binom{s-m-j}{n-1}(\lambda_2x_{-\beta_{1}}(s-m-j+1-n)\cdot\\
& \qquad \qquad x_{2\beta_1+\b_2}(-n_2+j)\prod_{i=3}^tx_{\beta_i}(-n_i)\1\\
&
\qquad+\lambda_{3}x_{-\beta_{1}}(s-m-j+1-n)x_{\beta_{1}+\beta_{2}}(-n_{2})x_{\beta_1+\b_3}(-n_3+j)\prod_{i=4}^tx_{\b_i}(-n_i)\1+\cdots\\
\\
& \qquad +
\!\lambda_{t}x_{-\beta_{1}}\!(\!s\!-\!m\!-\!j\!+\!1\!-\!n\!)x_{\beta_{1}\!+\!\beta_{2}}(\!-\!n_{2}\!)x_{\b_3}(-n_3)\cdots
x_{\beta_{t-1}}(-n_{t-1})x_{\b_1+\b_t}(-n_t+j)\1 ).
\end{split}
\end{equation}
for some $u\in U$ by the induction assumption.

Similarly, we can prove that each term in the summation
$\sum_{j=0}^{N-1}*$ in \eqref{eq:fe16} lies in $U.$ As a result,
$$x_{\beta_{1}}(-m)x_{\beta_{2}}(-2)x_{\beta_{3}}(-n_{3})\cdots
x_{\beta_{t}}(-n_{t})\1\in U.$$ Continuing in this way, we can
prove that
$$x_{\beta_{1}}(-m)x_{\beta_{2}}(-n)x_{\beta_{3}}(-n_{3})\cdots
x_{\beta_{t}}(-n_{t})\1\in U$$ for any $m,n,n_{3},\cdots,n_{k}>0$.
This completes the proof of Case 1.

{\bf Case 2.} For any $1\leq i,j\leq t,
\;\beta_{i}+\beta_{j}\notin \Delta.$ We claim that  there exist
$1\leq i^{'},j^{'}\leq t$ such that
$\beta_{i^{'}}+\beta_{j^{'}}=0.$ Otherwise, $\b_i+\b_j\ne 0$ for
all $i,j.$ This implies that $\<\b_i,\b_j\>\geq 0$ for all $i,j$,
thus $\<\b_1,\sum_{j=2}^t\b_j\>\geq 0.$ On the other hand, since
$\sum_{j=2}^t\b_j=-\b_1$, we have $\<\b_1,\sum_{j=2}^t\b_j\><0,$ a
contradiction.

Without loss, we may assume $\beta_{1}+\beta_{2}=0$ by the
induction assumption. Let $w=x_{\beta_{3}}(-n_{3})\cdots
x_{\beta_{t}}(-n_{t})\1$ with $\beta_{3}+\cdots+\beta_{t}=0.$ Take
$n=1$ in \eqref{eq:fel} and apply to $w$ to produce
\begin{equation}\label{eq:tfe11}
\begin{split}
& \qquad \quad (x_{\beta_{1}}(-m)x_{-\beta_{1}}(-1)\1)_{s}w\\
& \qquad = \sum_{ j \geq 0}(-1)^{j}\binom{-m}{j}x_{\beta_{1}}(-m-j)x_{-\beta_{1}}(s+j)w\\
& \qquad \quad  -\sum_{j\geq
0}(-1)^{-m+j}\binom{-m}{j}x_{-\beta_{1}}(s-m-j)x_{\beta_{1}}(j)w.
\end{split}
\end{equation}
We let $s=-1$ and
consider the first summation of the right hand side of
\eqref{eq:tfe11}. If $j=0$, then $s+j=-1$, and if $j\geq 1$ then
$s+j\geq 0.$  Thus the identity \eqref{eq:tfe11} becomes
\begin{equation}\label{eq:tfe12}
\begin{split}
&  \qquad \quad (x_{\beta_{1}}(-m)x_{-\beta_{1}}(-1)\1)_{s}w\\
& \qquad =
x_{\beta_{1}}(-m)x_{-\beta_{1}}(-1)x_{\beta_{3}}(-n_{3})\cdots
x_{\beta_{t}}(-n_{t})\1+u
\end{split}
\end{equation}
for some $u\in U$ by the induction assumption. Since the LHS of
 \eqref{eq:tfe12} is always in $U$, we have
$x_{\beta_{1}}(-m)x_{-\beta_{1}}(-1)x_{\beta_{3}}(-n_{3})\cdots
x_{\beta_{t}}(-n_{t})\1\in U$ for any $m, n_{3},\cdots, n_{t}>0$.

Similarly, we take $s=-2$ in  \eqref{eq:tfe11} and
work with the first summation of the
right hand side of \eqref{eq:tfe11} to obtain
\begin{equation*}\label{eq:tfe13}
\begin{split}
&  \qquad \quad (x_{\beta_{1}}(-m)x_{-\beta_{1}}(-1)\1)_{s}w\\
& \qquad =
x_{\beta_{1}}(-m)x_{-\beta_{1}}(-2)x_{\beta_{3}}(-n_{3})\cdots
x_{\beta_{t}}(-n_{t})\1\\
& \qquad +
mx_{\beta_{1}}(-m-1)x_{-\beta_{1}}(-1)x_{\beta_{3}}(-n_{3})\cdots
x_{\beta_{t}}(-n_{t})\1+u^{'}
\end{split}
\end{equation*}
for some $u^{'}\in U$ by the induction assumption. Since we have already
proved that
$$x_{\beta_{1}}(-m)x_{-\beta_{1}}(-1)x_{\beta_{3}}(-n_{3})\cdots
x_{\beta_{t}}(-n_{t})\1\in U$$ for any $m, n_{3},\cdots, n_{t}>0$,
we get
$$x_{\beta_{1}}(-m)x_{-\beta_{1}}(-2)x_{\beta_{3}}(-n_{3})\cdots
x_{\beta_{t}}(-n_{t})\1\in U$$ for any $m, n_{3},\cdots, n_{t}>0$.
Continuing in this way, we can prove that
$$x_{\beta_{1}}(-m)x_{-\beta_{1}}(-n)x_{\beta_{3}}(-n_{3})\cdots
x_{\beta_{t}}(-n_{t})\1\in U$$ for any $m,n,n_{3},\cdots,n_{k}>0$.
This completes the proof of Case 2.

\end{proof}

Next we discuss some  automorphisms of vertex operator algebras
$V(k,0)$ and $V(k,0)(0)$ for later purpose. It is well known that
the automorphism group $Aut(V(k,0))$ is isomorphic to the
automorphism group $Aut(\g).$ In fact, if $\sigma\in Aut(\g)$,
then $\sigma$ lifts to an automorphism of $V(k,0)$ in the
following way:
$$\sigma(x_1(-n_1)\cdots x_s(-n_s)\1)=(\sigma x_1)(-n_1)\cdots (\sigma
x_s)(-n_s)\1$$ for $x_i\in \g$ and $n_i>0.$ Let $W(\g)$ be the
Weyl group of $\g.$ Then $W(\g)$ can naturally be regarded as a
subgroup of $Aut (\g)$ \cite{Hu}. It is easy to see that if
$\sigma(\h)=\h$, then $\sigma (V(k,0)(0))=V(k,0)(0)$ and the
restriction of $\sigma$ to $V(k,0)(0)$ gives an automorphism of
$V(k,0)(0).$ In particular, any Weyl group element gives an
automorphism of $V(k,0)(0).$ This fact will be used in later
sections.

\section{Vertex operator algebra $N(\g,k)$ }
\label{walgebra}
\def\theequation{3.\arabic{equation}}
\setcounter{equation}{0}

Let $V_{\wh}(k,0)$ be the vertex operator subalgebra of $V(k,0)$
generated by $h(-1)\1$ for $h\in \mathfrak h$ with the Virasoro
element
%\begin{equation}\label{eq:omega_gamma}
$$\omega_{\mathfrak h} = \frac{1}{2k}
\sum_{i=1}^{l}h_i(-1)h_i(-1)\1$$
%\end{equation}
of central charge $l$, where $\{h_1,\cdots h_l\}$ is an
orthonormal basis of $\mathfrak h$ as before. For $\lambda\in
{\mathfrak h}^*$, let $M_{\wh}(k,\lambda)$ denote irreducible
highest weight module for $\wh$ with a highest weight vector
$v_\lambda$ such that $h(0)v_\lambda = \lambda(h) v_\lambda$ for
$h\in \mathfrak h.$ Then $V_{\wh}(k,0)$ is identified with
$M_{\wh}(k,0).$

Recall $V(k,0)(\lambda)$ from  Section 2. Both $V(k,0)$ and $V(k,0)(\lambda)$, $\lambda \in Q$
are completely reducible $V_{\wh}(k,0)$-modules. That is,
\begin{equation}\label{eq:dec-Heisenberg}
V(k,0) = \oplus_{\lambda\in Q} M_{\wh}(k,\lambda) \otimes
N_\lambda,
\end{equation}
\begin{equation}\label{eq:dec-Heisenberg1}
V(k,0)(\lambda)= M_{\wh}(k,\lambda) \otimes N_\lambda
\end{equation}
where
\begin{equation*}
N_\lambda = \{ v \in V(k,0)\,|\, h(m)v =\lambda(h)\delta_{m,0}v
\text{ for }  h\in \mathfrak h, m \ge 0\}
\end{equation*}
is the space of highest weight vectors with highest weight $\lambda$ for
$\wh.$

Note that $N(\g,k)=N_0$ is the commutant \cite[Theorem 5.1]{FZ} of
$V_{\wh}(k,0)$ in $V(k,0)$. The commutant $N(\g,k)$ is a vertex
operator algebra with the Virasoro vector $\omega =
\omega_{\mraff} - \omega_{\mathfrak h}$ whose central charge is
$\frac{k\dim \g}{k+h^{\vee}}-l.$

Recall from Section 2, the 3-dimensional subalgebra $\g^\alpha$
for $\alpha\in \Delta_+.$ Then the restriction $\<,\>_{\g^\alpha}$
of the bilinear form $\<,\>$ to $\g^\a$ is equal to
$\frac{\<\theta,\theta\>}{\<\alpha,\alpha\>}(,)$, where $(,)$ is
the standard nondegenerate symmetric invariant bilinear form on
$\g^\a$ such that $(h_\a,h_\a)=2.$ As a result, $V(k,0)$ is a
module for $\widehat{\g^\a}=\g^\a\otimes \C[t,t^{-1}]\oplus \C K$
of level $k_{\alpha}=\frac{\<\theta,\theta\>}{\<\alpha,\alpha\>}k$
as we regard $V(k,0)$ as a module for the subalgebra
$\widehat{\g^\a}$ of $\widehat\g.$ In other words, $V(k,0)$ is a
$\widehat{\g^\a}$-module of level $2k$ or $3k$ if $\alpha$ is a
short root.

Following \cite{DLY2}, we let
\begin{equation}\label{eq:w3}
\begin{split}
\omega_{\alpha} &=\frac{1}{2k(k+2)}( -k h_\alpha(-2)\1
-h_\alpha(-1)^{2}\1 +2kx_{\alpha}(-1)x_{-\alpha}(-1)\1),
\end{split}
\end{equation}
\begin{equation}\label{eq:W3}
\begin{split}
W_{\alpha}^3 &= k^2 h_\alpha(-3)\1 + 3 k h_\alpha(-2)h_\alpha(-1)\1
+
2h_\alpha(-1)^3\1 -6k h_\alpha(-1)x_{\alpha}(-1)x_{-\alpha}(-1)\1 \\
& \quad +3 k^2x_{\alpha}(-2)x_{-\alpha}(-1)\1 -3
k^2x_{\alpha}(-1)x_{-\alpha}(-2)\1
\end{split}
\end{equation}
if $\alpha\in\triangle_{+}$ is a long root. If $\alpha$ is a short
root, we also define $\omega_{\alpha}, W^3_{\alpha}$ as in
\eqref{eq:w3} and \eqref{eq:W3} by replacing $k$ by $k_{\alpha}.$

Let $\widehat{P}_{\alpha}$ be the vertex operator subalgebra of
$N(\g,k)$ generated by $\omega_{\alpha}$ and $W_\alpha^3.$ Then
$\widehat{P}_\alpha$ is isomorphic to the $W$-algebra $W(2,3,4,5)$
\cite{BEHHH} by \cite[Theorem 3.1]{DLY3} with $k$ replacing by
$k_\a$ or $N(sl_2,k_{\a}).$

The first main theorem of this paper is about the generators of $N(\g,k)$.
\begin{thm}\label{generator2} The vertex operator
algebra $N(\g,k)$ is generated by $\dim \g-l$ vectors $\omega_{\alpha}$
and $W_{\alpha}^3$ for $\alpha\in\Delta_{+}$. That is, $N(\g,k)$ is generated
by subalgebras $\widehat{P}_{\a}$ for $\a\in \Delta_+.$
\end{thm}

\begin{proof} We first prove that
$V(k,0)(0)$ is generated by vectors $\alpha_{i}(-1)\1$,
$\omega_{\alpha}$ and $W_{\alpha}^3$ for $i=1,\cdots,l$ and
$\alpha\in\Delta_{+}$. In fact, let $U$ be the vertex operator
subalgebra generated by  $h(-1)\1$, $\omega_{\alpha}$ and
$W_{\alpha}^3$ for $h\in\h$ and $\alpha\in\Delta_{+}$. Then
 $x_{-\alpha}(-1)x_{\alpha}(-1)\1\in U$ and $\omega_{\mraff}\in U$. Moreover, from the expression
 of $W_{\alpha}^3$, we see that $x_{-\alpha}(-1)x_{\alpha}(-2)\1 -
x_{-\alpha}(-2)x_{\alpha}(-1)\1\in U$. Set
$L_{\mraff}(n)=({\omega_{\mraff}})_{n+1}$, we have
\begin{equation*}
[L_{\mraff}(m), a(n)]=-na(m+n)
\end{equation*}
for $m,n\in\Z, a\in \g$. Thus,
\begin{equation*}
\begin{split}
L_{\mraff}(-1)x_{-\alpha}(-1)x_{\alpha}(-1)\1=x_{-\alpha}(-2)x_{\alpha}(-1)\1+x_{-\alpha}(-1)x_{\alpha}(-2)\1\in
U.
\end{split}
\end{equation*}
Together with $x_{-\alpha}(-1)x_{\alpha}(-2)\1 -
x_{-\alpha}(-2)x_{\alpha}(-1)\1\in U$, we get
$x_{-\alpha}(-2)x_{\alpha}(-1)\1 \in U$, and so $U$ is equal to
$V(k,0)(0)$ by Theorem \ref{generator1}.

Next we show that $\omega_{\alpha}, W_{\alpha}^3\in N(\g,k)$ for
$\alpha\in \Delta.$ Since $\la h_\alpha,h_\alpha\ra \ne 0,$ we
have decomposition $\mathfrak h=\C h_\alpha \oplus (\C
h_\alpha)^{\bot}$, where $(\C h_\alpha)^{\bot}$ is the orthogonal
complement of $\C h_\alpha$ with respect to $\la,\ra.$ From
\cite{DLY3}, we know that
$h_\alpha(n)\omega_{\alpha}=h_\alpha(n)W_{\alpha}^3=0$ for $n\geq
0.$ If $u\in (\C h_\alpha)^{\bot},$ we clearly have
$u(n)\omega_{\alpha}=u(n)W_{\alpha}^3=0$ for $n\geq 0.$ This
implies that $\omega_{\alpha}, W_{\alpha}^3\in N(\g,k).$

Notice that $Y(u,z_1)Y(v,z_2 )=Y(v,z_2)Y(u,z_1)$ for $u \in
M_{\wh}(k,0)$ and $v \in N(\g,k)$. Since $V(k,0)(0)= M_{\wh}(k,0)
\otimes N(\g,k),$ $h(-1)\1 \in M_{\wh}(k,0)$ for $h\in \h$, and
$\omega_{\alpha}, W_{\alpha}^3 \in N(\g,k)$ for
$\alpha\in\Delta_{+}$, we conclude that $N(\g,k)$ is generated by
$\omega_{\alpha}, W_{\alpha}^3 $ for $\alpha\in\Delta_{+}$.
\end{proof}

\begin{rmk} Using the $Z$-algebra introduced and studied in \cite{LP}
and \cite{LW1}-\cite{LW2}, we can rewrite $\omega_\a$ and
$W_\alpha^3$ in terms of $Z$-operators $Z_{\a}(m)$ and
$Z_{-\a}(n).$ It is not too hard to see  that
$\omega_\a=a_{\alpha}Z_{\a}(-1)Z_{-\a}(-1)\1$ and $W_\a^3=b_\a
Z_\a(-2)Z_{-\a}(-1)\1+c_{\alpha}Z_{-\a}(-2)Z_{\a}(-1)\1$ for some
constants $a_{\a}, b_{\a}, c_\alpha\in \C.$ One could determine
these constants explicitly using the definition of $Z$-operators.
\end{rmk}

\begin{rmk} The vertex operator algebra $N(\g,k)$ and its quotient $K(\g,k)$
are of moonshine type. That is, their weight zero subspaces are 1-dimensional
and weight one subspaces are zero.
\end{rmk}

\begin{rmk} We have already known that\underline{}
each $\omega_\a$ is a Virasoro element and how to compute the Lie
brackets $[Y(\omega_\a, z_1), Y(W_\a^3,z_2)],[Y(W^3_\a, z_1),
Y(W_\a^3,z_2)]$ for $\a\in \Delta_+.$ It is important to calculate
the Lie brackets for vertex operators associated to vectors in
different $\widehat{P_\a}.$ This will be done in a sequel to this
paper where the representation theory will be investigated.
\end{rmk}

Following the discussion given at the end of Section 2, we see that
any Weyl group element gives an automorphism of $N(\g,k).$

\section{Parafermion vertex operator algebras $K(\g,k)$
}\label{Sect:maximal-ideal-tI}
\def\theequation{4.\arabic{equation}}
\setcounter{equation}{0}

It is well known that vertex operator algebra $V(k,0)$ has a
unique maximal ideal $\J$ generated by a weight $k+1$ vector
$x_{\theta}(-1)^{k+1}\1$ \cite{K}, where $\theta$ is the highest
root of $\g$. The quotient vertex operator algebra $L(k,0) =
V(k,0)/\J$ is a simple, rational  vertex operator algebra
associated to affine Lie algebra $\widehat{\g}.$ Again, the
Heisenberg vertex operator algebra $V_{\wh}(k,0)$ generated by
$h(-1)\1$ for $h\in \mathfrak h$ is a simple subalgebra of
$L(k,0)$ and $L(k,0)$ is a completely reducible
$V_{\wh}(k,0)$-module. We have a decomposition
\begin{equation}
L(k,0) = \oplus_{\lambda\in Q} M_{\wh}(k,\lambda) \otimes
K_\lambda
\end{equation}
as modules for $V_{\wh}(k,0)$, where
\begin{equation*}
K_\lambda = \{v \in L(k,0)\,|\, h(m)v =\lambda(h)\delta_{m,0}v
\text{ for }\; h\in {\mathfrak h},
 m \ge 0\}.
\end{equation*}
Set $K(\g,k)=K_0.$ Then $K(\g,k)$ is the commutant of
$V_{\wh}(k,0)$ in $L(k,0)$ and called the parafermion vertex
operator algebra associated to the irreducible highest weight
module $L(k,0)$ for $\widehat{\g}.$ As we mentioned in the
introduction, $K(\g,k)$ are conjectrued to be rational,
$C_2$-cofinite vertex operator algebras.

As a $V_{\wh}(k,0)$-module, $\J$ is completely reducible. From
\eqref{eq:dec-Heisenberg},
\begin{equation*}
\J = \oplus_{\lambda\in Q} M_{\wh}(k,\lambda) \otimes (\J \cap
N_\lambda).
\end{equation*}
In particular, $\tI = \J \cap N(\g,k)$ is an ideal of $N(\g,k)$ and $K(\g,k)
\cong N(\g,k)/\tI$. The same proof as \cite[Lemma 3.1]{DLY2}, we know
that $\tI$ is the unique maximal ideal of $N(\g,k).$ Thus $K(\g,k)$ is a
simple vertex operator algebra.

We still use $\omega_{\mraff}$, $\omega_{\mathfrak h}$,
$\omega_{\alpha}$, $W_{\alpha}^3$ to denote their
images in $L(k,0) = V(k,0)/\J$.
\begin{rmk} In the case $k=1$, it follows from the construction
of $L(1,0)$ \cite{FK}, \cite{FLM} that $\omega=0$ and $K(\g,k)=\C.$
\end{rmk}

The following result is a direct consequence of Theorem
\ref{generator2}.

\begin{thm}\label{generator3} The simple vertex operator algebra $K(\g,k)$
is generated by $\omega_{\alpha}$, $W_{\alpha}^3$ for $\alpha\in \Delta_{+}$.
\end{thm}

Next, we study the ideal $\tI$ of $N(\g,k)$ in detail. The vector
$x_{\theta}(-1)^{k+1}\1\notin N(\g,k)$. From \cite[Theorem
3.2]{DLY2} we know that
$h_\theta(n)x_{-\theta}(0)^{k+1}x_{\theta}(-1)^{k+1}\1=0$ for
$n\geq 0.$ It is clear that if $h\in \mathfrak h$ satisfying $\la
h_\theta,h\ra=0$, then
$h(n)x_{-\theta}(0)^{k+1}x_{\theta}(-1)^{k+1}\1=0$ for $n\geq 0.$
So we have proved the following

\begin{lem}\label{l1} $x_{-\theta}(0)^{k+1}x_{\theta}(-1)^{k+1}\1\in \tI.$
\end{lem}
Furthermore, we have

\begin{prop}\label{Conj:ideal-generator}
 The maximal ideal $\tI$ of $N(\g,k)$ is generated by $x_{-\theta}(0)^{k+1}x_{\theta}(-1)^{k+1}\1$.

\end{prop}

\begin{proof}
The proof is similar to that of \cite[Theorem 4.2 (1)]{DLY3}. Recall
$\g^{\theta}=\C x_{\theta}+\C h_\theta+\C x_{-\theta}$ is the
subalgebra of $\g$ isomorphic to $sl_2.$  $V(k,0)$ is an
$\g^{\theta}$-module where $a \in \g^{\theta}$ acts as $a(0)$. Each
weight subspace of the vertex operator algebra $V(k,0)$ is a finite
dimensional $\g^\theta$-module and $V(k,0)$ is completely reducible
as a module for $\g^\theta$. Consider the $\g^\theta$-submodule $X$
of $V(k,0)$ generated by $x_{\theta}(-1)^{k+1}\1$. Since
$x_{\theta}(0)x_{\theta}(-1)^{k+1}\1 =0$ and
$h_\theta(0)x_{\theta}(-1)^{k+1}\1 = 2(k+1)x_{\theta}(-1)^{k+1}\1,$
$x_{\theta}(-1)^{k+1}\1$ is a highest weight vector with highest
weight $2(k+1)$ for $\g^\theta$. Then $X$ is an irreducible
$\g^\theta$-module with basis
$x_{-\theta}(0)^{i}x_{\theta}(-1)^{k+1}\1$, $0 \le i \le 2(k+1)$
from the representation theory of $sl_2$. This implies that the
ideal $\J$ of the vertex operator algebra $V(k,0)$ can be generated
by any nonzero vector in $X$. In particular, $\J$ is generated by
$x_{-\theta}(0)^{k+1}x_{\theta}(-1)^{k+1}\1$. Then $\J$ is spanned
by $u_n x_{-\theta}(0)^{k+1}x_{\theta}(-1)^{k+1}\1$ for $u \in
V(k,0)$ and $n \in \Z$ by \cite[Corollary 4.2]{DM} or
\cite[Proposition 4.1]{L}.

Since $v_mu\in V(k,0)(\lambda+\mu)$ for $v\in V(k,0)(\lambda),$
$u\in V(k,0)(\mu),$ $\lambda,\mu\in Q$ and $m\in\Z,$ we see that
$u_nx_{-\theta}(0)^{k+1}x_{\theta}(-1)^{k+1}\1 \in \J \cap
V(k,0)(0)$ if and only if $u \in V(k,0)(0)$. Let $u = v\otimes
w\in V(k,0)(0) = M_{\wh}(k,0) \otimes N(\g,k)$ with $v \in
M_{\wh}(k,0)$ and $w\in N(\g,k)$. Then $Y(u,z)=Y(v,z)\otimes
Y(w,z)$ acts on $M_{\wh}(k,0) \otimes N(\g,k)$. As a result, we
have that $\tI$ is spanned by $w_n
x_{-\theta}(0)^{k+1}x_{\theta}(-1)^{k+1}\1$ for $w \in N(\g,k)$
and $n \in \Z$. That is, the ideal $\tI$ of the vertex operator
algebra $N(\g,k)$ is generated by
$x_{-\theta}(0)^{k+1}x_{\theta}(-1)^{k+1}\1$. The proof is
complete.
\end{proof}

For $\alpha\in \Delta_+$, we let $P_{\alpha}$ be the vertex
operator subalgebra of $K(\g,k)$ generated by $\omega_{\alpha}$
and $W_\alpha^3.$ Then $P_\alpha$ is a quotient of
$\widehat{P}_{\alpha}.$ A natural question is whether or not
$P_\alpha$ is a simple vertex operator algebra. For this purpose,
we recall our discussion earlier on the automorphisms of vertex
operator algebra $V(k,0)$ and $N(\g,k).$ That is, any Weyl group
element gives an automorphism of $V(k,0)$ and $N(\g,k).$

Clearly, any automorphism $\sigma$ of $V(k,0)$ induces an
automorphism of $L(k,0)$ as $\sigma$ maps the unique maximal ideal
$\J$ to $\J.$ If $\sigma\in W(\g)$, then $\sigma$ preserves the
unique maximal ideal $\tI$ and $\sigma$ gives an automorphism of
the parafermion vertex operator algebra $K(\g,k),$ Now let
$\alpha\in \Delta_+$ be a long root. Then there exists $\sigma\in
W(\g)$ such that $\sigma \theta=\alpha$ \cite{Hu}.  As a result,
$$\sigma(x_{-\theta}(0)^{k+1}x_{\theta}(-1)^{k+1}\1)=
ax_{-\a}(0)^{k+1}x_{\a}(-1)^{k+1}\1$$
for some constant $a.$
This implies from Lemma \ref{l1} that
$x_{-\a}(0)^{k+1}x_{\a}(-1)^{k+1}\1\in \tI.$ Using \cite[Theorem 4.2]{DLY3}
we obtain:
\begin{prop}\label{plast} For any long root $\alpha\in \Delta_+,$ the vertex operator
subalgebra $P_{\alpha}$ of $K(\g,k)$ is a simple vertex operator algebra
isomorphic to  the parafermion vertex operator algebra $K(sl_2,k)$ associated
to $sl_2.$
\end{prop}

We next deal with short roots $\a\in \Delta_+.$ As we mentioned
already that $V(k,0)$ is a level $k_\a$-module for the affine
algebra $\widehat{\g^\a}.$ We need a different method to prove the
following which is a generalization of Proposition \ref{plast}.
\begin{prop}\label{plast1} Let $\alpha\in \Delta_+.$ Then the vertex operator subalgebra $P_{\alpha}$
of $K(\g,k)$ is a simple vertex operator algebra isomorphic to
the parafermion vertex operator algebra $K(sl_2,k_\a)$ associated
to $sl_2.$
\end{prop}
\begin{proof} As in the proof of Proposition \ref{plast} we only
need to prove that
$$x_{-\a}(0)^{k_\a+1}x_{\a}(-1)^{k_\a+1}\1\in \tI.$$
Clearly, $L(k,0)$ is an integrable module for $\widehat{\g^\a}$ as
$x_{\a}(-1)$ is locally nilpotent on $L(k,0).$  In particular, the
vertex operator subalgebra $U$ of $L(k,0)$ generated by $\g^\a$ is
an integrable highest weight module. That is, $U$ is isomorphic to
$L(k_\a,0)$ associated to the affine algebra $\widehat{\g^\a}.$ As
a result, we have $x_{\a}(-1)^{k_\a+1}\1\in \J.$ It follows then
immediately that $x_{-\a}(0)^{k_\a+1}x_{\a}(-1)^{k_\a+1}\1\in
\tI,$ as desired.
\end{proof}

\begin{rmk} We expect from Proposition \ref{plast1} that the role of
$K(sl_2,k_\a)$ played in the theory of parafermion vertex operator
algebra is similar to the role of $sl_2$ played in the theory of
Kac-Moody Lie algebras. So a study of structural and
representation theory for $K(sl_2,k_\a)$ becomes  extremely
important for general parafermion vertex operator algebras.
\end{rmk}

\end{document}